\documentclass[12pt,a4paper,oneside]{amsart}
\usepackage{amsfonts,amsmath,amssymb,amsthm}
\usepackage{mathtools}
\usepackage{anysize}
\marginsize{2cm}{2cm}{1.5cm}{1.5cm}
\usepackage[utf8]{inputenc}
\usepackage[english]{babel}
\usepackage{cmap}
\usepackage{enumerate}
\usepackage{xcolor}
\usepackage{hyperref}
\hypersetup{
  colorlinks = true,
  urlcolor   = blue,
  linkcolor  = blue,
  citecolor  = red
}

\newcommand{\Href}[2]{\hyperref[#2]{#1~\ref{#2}}}
\newcommand{\st}{:\;}

\newcommand{\norm}[1]{\left\|#1\right\|}
\newcommand{\enorm}[1]{\left|#1\right|}
\newcommand{\normsch}[2][p]{\norm{#2}_{\mathrm{S}_{#1}}}

\newcommand{\iprod}[2]{\left\langle#1,#2\right\rangle}

\providecommand{\parenth}[1]{\left(#1\right)}
\providecommand{\braces}[1]{\left\{#1\right\}}

\providecommand{\abs}[1]{\lvert#1\rvert}
\providecommand{\card}[1]{\lvert#1\rvert}

\def\N{\mathbb N}
\def\R{\mathbb R}
\def\C{\mathbb C}
\newcommand{\hilbert}{\mathcal{H}}

\DeclareMathOperator{\mglname}{\varrho}
\newcommand{\mgl}[2]{\mglname\nolimits_{#1}\!\left( #2 \right)}
\newcommand{\mglx}[1]{\mglname\nolimits_{X}\!\!\!\:\left( #1 \right)}

\DeclareMathOperator{\mconame}{\delta}
\newcommand{\mco}[2]{\mconame\nolimits_{#1}\!\left( #2 \right)}
\newcommand{\mcox}[1]{\mconame\nolimits_{X}\!\left( #1 \right)}

\newcommand{\ball}[1]{\mathbf{B}^{#1}}
\newcommand{\di}{\,\mathrm{d}}
\newcommand{\id}{\mathrm{Id}}

\theoremstyle{theorem}
\newtheorem{thm}{Theorem}[section]
\newtheorem{prp}{Proposition}[section]
\newtheorem{lem}{Lemma}[section]

\newtheorem{conj}{Conjecture}[section]

\theoremstyle{definition}
\newtheorem{dfn}{Definition}[section]
\newtheorem{rem}[dfn]{Remark}

\def\diam{\mathop{\rm diam}}
\newcommand{\dist}[2]{\operatorname{dist}\!\left( #1, #2 \right)}%

\DeclareMathOperator{\Enc}{Enc}
\DeclareMathOperator{\Dec}{Dec}

\newtheorem{Problem}{Problem}[section]

\title{No-dimensional results of combinatorial convexity. Dimension strikes back}

\author{Grigory Ivanov\address{Grigory Ivanov: 
Pontifícia Universidade Cat\'olica do Rio de Janeiro \\
Departamento de Matematica,
Rua Marquês de São Vicente, 225\\
Edif{\'i}cio Cardeal Leme, sala 862,
22451-900 G{\'a}vea, Rio de Janeiro, Brazil}
\email{grimivanov@gmail.com}}
\thanks{The author is supported by Projeto Paz and Coordenacao de Aperfeicoamento de Pessoal de Nivel Superior - Brasil (CAPES) - 23038.015548/2016-06}

\date{\today}

\begin{document}
\begin{abstract}
We discuss no-dimensional (approximate) versions of Carath\'eodory's and Helly's theorems.
Our goal is to draw attention to open problems and potential applications related to these results.
We survey recent progress and pose several questions.
We also point out a simple way to ``bring the dimension back into the picture'': by combining no-dimensional statements with dimension-dependent norm comparisons, one can transfer problems in $\ell_1^d$, $\ell_\infty^d$, and Schatten classes $S_1, S_\infty$ to nearby $\ell_p^d$ or $S_p$ spaces with better geometry.

As elementary applications, we obtain
a weak additive analogue of the Johnson--Lindenstrauss flattening lemma,
local-to-global estimates for Chebyshev regression over the $\ell_1$ ball,
and a local-to-global guarantee for quantum feasibility from locally consistent linear measurements.
\end{abstract}

\maketitle
\section{Introduction}

The theorems of Helly~\cite{helly1923mengen} and Carath\'eodory~\cite{caratheodory1911variabilitatsbereich} are two celebrated results that form part of the foundation of modern convexity and its adjacent areas.
At a basic level, these theorems reveal fundamental combinatorial and geometric properties of convex sets in finite-dimensional linear spaces and, in particular, they can be used to characterize the dimension of the ambient space.
Both results have given rise to extensive and largely independent lines of research; we refer to~\cite{barany2022helly, barany2021combinatorial} for recent surveys.

In a breakthrough paper~\cite{adiprasito2020theorems}, Adiprasito, B\'ar\'any, Mustafa, and Terpai obtained \emph{no-dimensional} (or \emph{approximate}) versions of Helly's and Carath\'eodory's theorems in Euclidean spaces.
These results establish a striking connection between combinatorial convexity and approximation theory: instead of exact intersection/containment statements whose quantitative form depends on the dimension,
one gets approximate conclusions controlled by a single natural parameter.
Motivated by potential applications to algorithmic questions (such as linear sketching and local-to-global guarantees in optimization), it is natural to ask for analogous results beyond Euclidean spaces.

The goal of this note is to bring attention to this circle of problems.
We formulate several open questions and record some straightforward corollaries of the no-dimensional results in Banach spaces.
At the same time, we make one additional tweak and ``return the dimension to the picture'': we combine no-dimensional results with a standard substitution trick from Banach space theory,
which allows one to pass from $\ell_1$-type or $\ell_\infty$-type geometry to an appropriately chosen $\ell_p$-space with desired geometric properties.

We will treat the needed Banach space facts as a black box.
For convenience, we recall all relevant definitions and inequalities in \Href{Section}{sec:banach_space_likbez}, in order to keep the paper self-contained.
Then we explain in detail what we mean by no-dimensional versions of Helly's and Carath\'eodory's theorems, survey recent results on the topic, and formulate several open problems.
Finally, we describe a standard substitution trick and show that a number of ``local-to-global'' statements become straightforward corollaries of the no-dimensional theory.

\subsection{Applications of no-dimensional results}

In many applications \cite{Garofalakis2013,Wang2016} one stores or transmits large objects
(signals, images, fields), while the downstream task uses only pairwise similarity queries,
often measured by the squared $L_2$-distance (equivalently, the $L_2$-energy).
The following theorem shows that, for a fixed collection of $n$ signals, one can replace each
signal by a short ``fingerprint'' in $\R^k$ for sufficiently small $k$, while preserving all
pairwise squared $L_2$-distances up to additive error~$\varepsilon$.

The result can be viewed as a weak additive variant of the celebrated Johnson--Lindenstrauss
flattening lemma \cite{johnson1984extensions}; see also \cite{vershynin2018high} for modern
expositions. While the classical Johnson--Lindenstrauss lemma is typically proved by a random
projection argument, our proof is based on a deterministic construction arising from a
no-dimensional Carath\'eodory-type lemma. This provides a simple illustration of the method.

\begin{thm}[Communication-efficient similarity search for signals]\label{thm:jl_signals_receiver}
Fix $d\ge 1$ and $n \ge 8,$ and let $f_1,\dots,f_n\colon [0,1]^d\to[-1,1]$ be continuous.
A \emph{sketching protocol} consists of an encoding map
\[
\Enc\colon C([0,1]^d)\to\R^k
\]
and a decoding rule
\[
\Dec\colon \R^k\times \R^k\to\R,
\]
where $\Dec(\Enc(f_i),\Enc(f_j))$ is interpreted as an estimate of the squared $L_2$-distance
\[
D(i,j):=\int_{[0,1]^d}\abs{f_i(x)-f_j(x)}^2\,\di x.
\]

Then for every $\varepsilon\in(0,1)$ there exist $k\le C\,\frac{\ln n}{\varepsilon^2}$ and maps $\Enc,\Dec$
such that for all $i,j\in[n]$,
\[
\abs{\Dec(\Enc(f_i),\Enc(f_j)) - D(i,j)} \le \varepsilon.
\]
Moreover, one may take $\Dec(u,v)=\enorm{u-v}^2$ (i.e.\ decoding is just squared Euclidean distance $\enorm{\cdot}^2$) and $\Enc$ to be linear.
\end{thm}

We emphasize that this result is much weaker than the Johnson--Lindenstrauss flattening lemma.
First, we only obtain an additive approximation to the squared $L_2$-distance. Second, the range
of the functions is a fixed interval, so the supremum norm is uniformly bounded (and we exploit
this in the proof). Third, the functions are defined on a domain of measure one, which is also
important for our argument since we work on a probability space.

The main point is that the statement admits a deterministic proof based on a no-dimensional
Carath\'eodory lemma.

Also, we obtain several local-to-global estimates for $\ell_1$-minimization problems using no-dimensional Helly-type results.
We write $[m]:=\braces{1,\dots,m}$ for $m\in\N$, and we equip $\R^d$ with the standard $\ell_p$ norm $\norm{\cdot}_p$.
We denote by $\ell_p^d:=(\R^d,\norm{\cdot}_p)$ the resulting normed space and by $\ball{}_{\ell_p^d}$ its unit ball.

\medskip
One basic example is \emph{Chebyshev regression} (also known as \emph{minimax absolute residual fitting} or $\ell_\infty$-regression), the classical problem of fitting a linear model by minimizing the worst absolute residual:
\[
\min_{x\in\R^d}\ \max_{i\in[m]} \abs{\iprod{a_i}{x}-b_i}.
\]
It admits an immediate linear programming formulation by introducing an auxiliary variable $t$ and enforcing the two-sided constraints
$-t\le \iprod{a_i}{x}-b_i\le t$ for all $i\in[m]$; see, e.g., \cite{boyd2004convex,ben2001lectures}.
From an algorithmic viewpoint, this connects Chebyshev regression to the general toolbox of convex optimization and linear programming, while its geometry is governed by small active constraint sets, a perspective formalized in the LP-type framework and its randomized algorithms \cite{clarkson1995vegas,matouvsek1992subexponential}.
More recently, Chebyshev/$\ell_\infty$ guarantees have also been studied via modern randomized linear algebra and ``sketch-and-solve'' methods, where one seeks fast approximate solutions together with explicit $\ell_\infty$ control \cite{price2017fast,song2023nearly}.
In the present work, we pursue a different direction: using no-dimensional Helly-type results to derive \emph{local-to-global} approximation guarantees from $k$-wise feasibility information, with explicit dimension dependence.

Our first local-to-global statement shows that approximate global feasibility over the $\ell_1$ ball follows from approximate feasibility on all $k$-subcollections of constraints, with an explicit  dependence on the dimension.

\begin{thm}[Local-to-global Chebyshev regression]\label{thm:chebyshev_ball}
Let $a_1,\dots,a_m\in [-1,1]^d$ and $b_1,\dots,b_m\in \R$.
Fix $k\in[m]$ and parameters $R\ge 1$ and $r\ge 0$.
Suppose that for every subset $J\subset[m]$ with $\card{J}=k$ there exists a point $x_J\in R\ball{}_{\ell_1^d}$ such that
\[
\max_{j\in J}\abs{\iprod{a_j}{x_J}-b_j}\ \le\ r.
\]
Then there exists a point $x\in eR\ball{}_{\ell_1^d}$ such that
\[
\max_{i\in[m]}\abs{\iprod{a_i}{x}-b_i}
\ \le\
r + 21\,R\,\sqrt{\frac{\ln d}{k}}.
\]
\end{thm}

\medskip
Our second application of no-dimensional Helly-type results is motivated by feasibility and consistency questions in quantum information theory.
A prototypical instance is the \emph{consistency of local density matrices} (also known as the \emph{quantum marginal problem}), which asks whether a collection of local constraints can arise from a single global quantum state; this problem is known to be computationally intractable in general; it was first shown to be in QMA and QMA-complete under Turing reductions by Liu \cite{liu2006consistency}, and subsequently proven to be QMA-complete under Karp reductions by Broadbent and Grilo \cite{broadbent2022qma}.
From a geometric viewpoint, the natural ambient spaces are Schatten classes and their convexity properties, and we refer to \cite{aubrun2017alice} for background on the interface between asymptotic geometric analysis and quantum information.

Fix $d\ge 1$ and let $\mathbb{F}\in\{\R,\C\}$.
Write $\mathcal{M}_d:=\mathbb{F}^{d\times d}$ for the space of $d\times d$ matrices over $\mathbb{F}$.
We say that $A\in\mathcal{M}_d$ is \emph{self-adjoint} if $A=A^\ast$, where $A^\ast$ denotes the transpose when $\mathbb{F}=\R$
and the conjugate transpose when $\mathbb{F}=\C$.
If $\mathbb{F}=\C$, let
\[
\mathcal{H}_d:=\{A\in\C^{d\times d}:\ A=A^\ast\}
\]
denote the  vector space of Hermitian matrices.

For $1\le p\le\infty$ we write $\normsch[p]{A}$ for the Schatten $p$-norm (the $\ell_p$-norm of the singular value vector),
and we equip $\mathcal{M}_d$ with the Frobenius (Hilbert--Schmidt) inner product
\[
\iprod{A}{B}:=\mathrm{Tr}(A^\ast B).
\]

A matrix $\rho\in\mathcal{M}_d$ is called \emph{positive semidefinite}, written $\rho\succeq 0$, if it is self-adjoint and has
nonnegative spectrum. A \emph{density matrix} is a positive semidefinite matrix of unit trace:
\[
\rho\succeq 0,
\qquad
\mathrm{Tr}(\rho)=1.
\]

Given self-adjoint measurement operators $A_1,\dots,A_m\in\mathcal{M}_d$ and target values $b_1,\dots,b_m\in\R$,
we consider the feasibility problem of finding a density matrix whose linear measurements $\iprod{A_i}{\rho}$ are simultaneously
consistent with $b_i$ up to a tolerance level.
In this work we study a \emph{local-to-global} version: assuming that every subcollection of $k$ constraints is feasible up to error $t$,
we ask whether this implies the existence of a single density matrix that approximately satisfies all $m$ constraints.

\begin{thm}[Quantum feasibility from local consistency]\label{thm:quantum_psd_local_to_global}
Let $d\ge 3$ and $k\in[m]$. Let $A_1,\dots,A_m\in\mathcal{H}_d$ be Hermitian matrices satisfying
\[
\normsch[\infty]{A_i}\le 1 \qquad\text{for all } i\in[m],
\]
and let $b_1,\dots,b_m\in\R$. Fix $t\ge 0$.
Assume that for every subset $J\subset[m]$ with $\abs{J}=k$, there exists a density matrix
$\rho_J\succeq 0$ with $\mathrm{Tr}(\rho_J)=1$ such that
\[
\abs{\iprod{A_j}{\rho_J} - b_j}
\le t \qquad\text{for all } j\in J.
\]
Then there exists a density matrix $\rho\succeq 0$ with $\mathrm{Tr}(\rho)=1$ such that
\[
\abs{\iprod{A_i}{\rho} - b_i}
\ \le\
t + 21\,\sqrt{\frac{\ln d}{k}},
\]
for all $i\in[m]$.
\end{thm}

\section{Definitions and results from Banach space theory}\label{sec:banach_space_likbez}

Throughout, $X$ is a Banach space with norm $\norm{\cdot}$.
The modulus of convexity $\mcox{\cdot}\st [0,2]\to\R$ is defined by
\[
\mcox{\varepsilon}
:=\inf\braces{\,1-\norm{\frac{x+y}{2}}:\ \norm{x}=\norm{y}=1,\ \norm{x-y}\ge \varepsilon\,}.
\]
A Banach space is called \emph{uniformly convex} if $\mcox{\varepsilon}>0$ for every $\varepsilon\in(0,2]$.

The \emph{modulus of smoothness} $\mglx{\cdot}\st [0,\infty)\to\R$ is defined by
\[
\mglx{t}
:=\sup\braces{\,\frac{\norm{x+ty}+\norm{x-ty}}{2}-1:\ \norm{x}=\norm{y}=1\,}.
\]
A Banach space $X$ is \emph{uniformly smooth} if $\frac{\mglx{\tau}}{\tau}\to 0$ as $\tau\to 0$.

A comprehensive discussion of these moduli can be found in~\cite{DiestelEng, lindenstrauss2013classical}.
In what follows, we will only use quadratic estimates for the moduli in certain $L_p$ spaces.

\begin{prp}\label{prp:modulus_conv_smooth_l_p_second_order}
\begin{enumerate}
\item[\textup{(i)}] If $1<p\le 2$, then for every $\varepsilon\in[0,2]$,
\[
\mco{\ell_p}{\varepsilon}\ \ge\ \frac{p-1}{8}\,\varepsilon^2.
\]
\item[\textup{(ii)}] If $2\le p<\infty$, then for every $t\ge 0$,
\[
\mgl{\ell_p}{t}\ \le\ \frac{p-1}{2}\,t^2.
\]
\end{enumerate}
\end{prp}

The lower bound on the modulus of convexity is due to~\cite{hanner1956uniform}.
The corresponding estimate for the modulus of smoothness follows from the Lindenstrauss polar identity relating the moduli of convexity and smoothness~\cite{lindenstrauss1963modulus}.

\medskip

Fix $d\ge 1$ and let $\mathcal{M}_d$ be the space of real (or complex) $d\times d$ matrices.
For $1\le p\le\infty$ we denote by $\normsch[p]{A}$ the \emph{Schatten $p$-norm}, i.e.\ the $\ell_p$-norm of the singular value vector.
Abusing notation, we write $S_p$ for the Banach space $\big(\mathcal{M}_d,\normsch[p]{\cdot}\big)$.

\begin{prp}\label{prp:modulus_conv_smooth_schatten_second_order}
\begin{enumerate}
\item[\textup{(i)}] If $1<p\le 2$, then for every $\varepsilon\in[0,2]$,
\[
\mco{S_p}{\varepsilon}\ \ge\ \frac{p-1}{8}\,\varepsilon^2.
\]
\item[\textup{(ii)}] If $2\le p<\infty$, then for every $t\ge 0$,
\[
\mgl{S_p}{t}\ \le\ \frac{p-1}{2}\,t^2.
\]
\end{enumerate}
\end{prp}

The lower bound on the modulus of convexity follows from~\cite[Theorem~1]{ball1994sharp}.
The corresponding estimate for the modulus of smoothness again follows from the Lindenstrauss polar identity~\cite{lindenstrauss1963modulus}.
\section{No-dimensional Carath\'eodory theorem}\label{sec:nodim_caratheodory}

Since no-dimensional versions of Carath\'eodory's theorem have been studied extensively,
we introduce the notion needed for this paper.

\begin{dfn}\label{dfn:caratheodory_sequence}
We call a sequence $\braces{R_k(X)}_{k \in \N}$ a \emph{Carath\'eodory sequence} of a Banach space $(X,\norm{\cdot})$
if it has the following property:

\smallskip
\noindent
Whenever the convex hull of a subset $Q$ of the unit ball $\ball{}_X$ contains the origin,
there exist points $x_1,\dots,x_k\in Q$ such that
\[
\norm{\frac{x_1+\cdots+x_k}{k}}
\le R_k(X).
\]
\end{dfn}

\begin{dfn}\label{dfn:caratheodory_approx_property}
We say that a Banach space $(X,\norm{\cdot})$ has the \emph{Carath\'eodory approximation property}
if there exists a Carath\'eodory sequence $\braces{R_k(X)}_{k\ge 1}$ with $R_k(X)\to 0$ as $k\to\infty$.
\end{dfn}

The Carath\'eodory approximation property gives a quantitative rate of approximation by convex combinations of bounded vectors.

The estimate $R_k(\hilbert)=\frac{1}{\sqrt{k}}$ for a Hilbert space $\hilbert$ is folklore;
it appears, for instance, as an exercise in standard texts on the probabilistic method and discrete mathematics;
see also~\cite[Theorem~1.1]{adiprasito2020theorems}.
This phenomenon is often presented as a \emph{convexification effect of Minkowski summation} and is essentially the key step
in the celebrated Shapley--Folkman--Starr theorem~\cite{starr1969quasi} (see the survey~\cite{fradelizi2018convexification}).

In Banach space theory, the no-dimensional Carath\'eodory theorem is usually referred to as \emph{Maurey's lemma}~\cite{pisier1980remarques}.
It implies, in particular, that Banach spaces of non-trivial type have the Carath\'eodory approximation property (We define neither type nor cotype  of a Banach space, since we will not use the corresponding results).
The standard proof is probabilistic and follows directly from the definition of type.
Recently, spaces with the Carath\'eodory approximation property were characterized in~\cite[Theorem~3.1]{artstein2025b}
using the language of the ``convexification effect'';
a corresponding counterexample was obtained in~\cite[Theorem~3.10]{kadets2022connection}.

\begin{prp}\label{prp:no-dim_Caratheodory}
The following assertions are equivalent for a Banach space $X$:
\begin{enumerate}
\item $X$ has the Carath\'eodory approximation property;
\item $X$ is of non-trivial type.
\end{enumerate}
\end{prp}

Recently, the author~\cite{ivanov2021approximate} provided a greedy-algorithm proof of the no-dimensional Carath\'eodory theorem
in uniformly smooth Banach spaces, thereby derandomizing Maurey's lemma
(interestingly, the underlying idea was first suggested for a ``convexification effect'' result in~\cite{ivanov2012generalization}).
It was shown that
\[
R_k(X) \ =\ \frac{4 e^2}{k\,\rho_X^{-1}\parenth{1/k}},
\]
where $\rho_X^{-1}$ denotes the inverse function of the modulus of smoothness of $X$.

Combining this estimate with \Href{Proposition}{prp:modulus_conv_smooth_l_p_second_order}
and \Href{Proposition}{prp:modulus_conv_smooth_schatten_second_order}, we obtain the following.

\begin{lem}\label{lem:Caratheodory_sequence_l_p}
Let $2\le p<+\infty$.
Then the sequence $\braces{21\sqrt{\frac{p-1}{k}}}_{k\ge 1}$ is a Carath\'eodory sequence for $\ell_p$ and for $S_p$.
Moreover, whenever the convex hull of a subset $Q$ of the unit ball $\ball{}_X$ contains the origin,
there exists a greedy algorithm that selects points $x_1,\dots,x_k\in Q$ successively so that
\[
\norm{\frac{x_1+\cdots+x_k}{k}}
\le R_k(X)
\]
at each step.
\end{lem}

\begin{rem}
The authors of~\cite{adiprasito2020theorems} derived numerous consequences of the no-dimensional Carath\'eodory theorem
in the Euclidean case, including no-dimensional versions of Tverberg's theorem and weak $\varepsilon$-net theorems.
These results were later generalized to Banach spaces with non-trivial type in~\cite{ivanov2021no}.
The reader is invited to check whether the construction from~\cite[Theorem~3.10]{kadets2022connection}
shows that the corresponding statements fail for Banach spaces of trivial type.
No-dimensional colorful Tverberg-type and related problems have attracted significant attention recently
(see the survey~\cite{polyanskii2025no}).
Quite unexpectedly, an optimal version with a deterministic proof in the Euclidean case was obtained recently in~\cite{barabanshchikova2025tightcolorfulnodimensionaltverberg}.
Nevertheless, precise quantitative answers to the following problems are still missing.
\end{rem}

\begin{Problem}
Given a Banach space $X$ and integers $k\ge 2$ and $n\ge 2$, find the smallest value $\operatorname{Tv}(X,k,n)$
such that for any $n$-point sets $Q_1,\dots,Q_{n}\subset X$, there exists a pairwise disjoint system of transversals
$\{P_1,\dots,P_{k}\}$ of $Q_1,\dots,Q_n$ and a ball of radius
\[
\operatorname{Tv}(X,k,n)\cdot \max_{1\le i\le n} \diam Q_i
\]
that intersects the convex hull of each $P_i$, for $i\in[k]$.
\end{Problem}

\begin{conj}
For any fixed integer $k\ge 2,$
\[
\operatorname{Tv}(X,k,n)  \to 0
\]
as $n \to \infty$ if and only if $X$ is of non-trivial type.
\end{conj}

\begin{Problem}
Provide a greedy-algorithm proof of the inequality
\[
\operatorname{Tv}(X,k,n) \le \frac{C}{n\,\rho_X^{-1}\parenth{1/n}}
\]
for some universal constant $C$ and any uniformly smooth Banach space $X.$
\end{Problem}

\section{No-dimensional Helly theorem}\label{sec:nodim_helly}

In this section we formalize the ``approximate'' (no-dimensional) Helly-type property.
The guiding principle is as follows: if every small subfamily of convex sets has a common point in a bounded region, then certain neighborhoods of the sets in the whole family have a common point. Recall that by the $R$-neighborhood of a set $A$ we understand the set
\[
\braces{x  \in X \st \dist{x}{A} \le R}.
\]

\begin{dfn}\label{dfn:helly_sequence}
Let $X$ be a Banach space.
We call a sequence $\braces{r_k(X)}_{k \in \N}$ a \emph{Helly approximation sequence} of $X$ if it has the following property:

\smallskip
\noindent
For every finite family $\mathcal{F}$ of convex subsets of $X$, if every subfamily of $\mathcal{F}$ of size $k$ has a common point inside the unit ball,
then the $r_k(X)$-neighborhoods of all sets in $\mathcal{F}$ have a common point.
\end{dfn}

\begin{dfn}\label{dfn:helly_approx_property}
We say that a Banach space $(X,\norm{\cdot})$ has the \emph{Helly approximation property} if there exists a Helly approximation sequence $\braces{r_k(X)}_{k \in \N}$ such that
\[
r_k(X)\to 0 \qquad\text{as } k\to\infty.
\]
\end{dfn}

It was shown in \cite[Theorem~1.2]{adiprasito2020theorems} that $r_k(\hilbert)=\frac{1}{\sqrt{k}}$ for Hilbert spaces.
In \cite{Ivanov2025} the author extended this to uniformly convex Banach spaces, obtaining the bound
\[
r_k(X) \leq \max\left\{1,\, 2\left(\frac{2}{q C_X}\right)^{1/q}\right\}
\]
under the assumption that $\delta_X(\varepsilon) \geq C_X \varepsilon^q$ for some $C_X>0$ and $q\ge 2$.

Combining this estimate with \Href{Proposition}{prp:modulus_conv_smooth_l_p_second_order}
and \Href{Proposition}{prp:modulus_conv_smooth_schatten_second_order}, one obtains the following.

\begin{lem}[Explicit bound for $\ell_p$, $1<p\le 2$]\label{lem:rk_bound_lp}
Let $1<p\le 2$.
Then the sequence $\braces{\frac{4\sqrt{2}}{\sqrt{p-1}}\cdot \frac{1}{\sqrt{k}}}_{k\in \N}$ is a Helly approximation sequence for $\ell_p$ and for $S_p$.
\end{lem}

Much less is currently understood about no-dimensional Helly-type results, and we would like to pose several problems.

\begin{Problem}
Characterize Banach spaces with the Helly approximation property.
\end{Problem}

It would be interesting to understand whether Banach spaces with the Helly approximation property are precisely those of non-trivial cotype.
In particular, this would imply that the $\ln d$ factor can be replaced by a constant in \Href{Theorem}{thm:chebyshev_ball}
and \Href{Theorem}{thm:quantum_psd_local_to_global}, since $\ell_1$ has a non-trivial cotype.

\section{A trick for bringing dimension back}\label{sec:trick_dimension_back}

We have listed several no-dimensional results. All of them require geometric properties of the ambient Banach space (uniform convexity, uniform smoothness, non-trivial type, etc.).
The difficulty is that the extremal spaces $\ell_1$ and $\ell_\infty$ do not possess any of these properties.
However, once we \emph{fix the dimension}, we can exploit dimension-dependent norm comparisons to ``transfer'' a problem from $\ell_1^d$ or $\ell_\infty^d$ to a nearby $\ell_p^d$ space that does enjoy the required geometry, at the cost of a mild multiplicative factor.
The same idea applies verbatim to Schatten classes.

\subsection{Dimension-dependent norm comparisons}

Let $x=(x_1,\dots,x_d)\in\R^d$ and $1\le p<\infty$. Recall the standard inequalities
\[
\norm{x}_{\infty}\le \norm{x}_p \le d^{1/p}\norm{x}_{\infty},
\qquad
\norm{x}_p\le \norm{x}_1 \le d^{1-1/p}\norm{x}_p.
\]
Set
\[
p:=\ln d,
\qquad
p':=\frac{p}{p-1}=\frac{\ln d}{\ln d-1}.
\]

\begin{prp}\label{prp:l_p_trick_bounds}
Let $d\ge 3$ and define
\[
p:=\ln d,
\qquad
p':=\frac{\ln d}{\ln d-1}.
\]
Then the following statements hold:

\smallskip
\noindent\textup{(i) Vector case.}
For every $x\in\R^d$,
\[
\norm{x}_{\infty}\le \norm{x}_p \le e\,\norm{x}_{\infty},
\qquad
\norm{x}_{p'} \le \norm{x}_1 \le e\,\norm{x}_{p'}.
\]

\smallskip
\noindent\textup{(ii) Schatten case.}
For every matrix $A\in \mathcal{M}_d$,
\[
\normsch[\infty]{A}\le \normsch[p]{A}\le e\,\normsch[\infty]{A},
\qquad
\normsch[p']{A} \le \normsch[1]{A} \le e\,\normsch[p']{A}.
\]
\end{prp}

\subsection{How the trick is used}

The ideology is simple: \emph{substitute} a difficult space by a nearby space with better geometry, paying only a small (dimension-dependent) factor.

\begin{itemize}
\item For approximation problems in $\ell_\infty^d$ (or $S_\infty$), we pass to $\ell_p^d$ (or $S_p$) with $p=\ln d$.
By \Href{Proposition}{prp:l_p_trick_bounds}, this changes norms by at most a factor $e$, while $\ell_p^d$ and $S_p$ are uniformly smooth for $p<\infty$ and have explicit Carath\'eodory sequences (cf.\ \Href{Lemma}{lem:Caratheodory_sequence_l_p}).

\item For $\ell_1$-minimization (and its Schatten analogue), we use the dual substitution
\[
\ell_1^d \ \longrightarrow\ \ell_{p'}^d,
\qquad
S_1 \ \longrightarrow\ S_{p'},
\qquad
p'=\frac{\ln d}{\ln d-1},
\]
so that $p'\in(1,2]$ and the target space is uniformly convex with explicit bounds on the modulus of convexity.
This is precisely the regime where our no-dimensional Helly-type results apply.
\end{itemize}

Finally, we emphasize that this substitution trick is a general principle in finite dimensions. For any $d$-dimensional Banach space $X$, its norm can be approximated arbitrarily well by a uniformly convex and uniformly smooth norm. This follows, for instance, from results on the approximation of convex bodies by smooth ones (see, e.g., \cite{schneider2014convex}). Consequently, our approach can, in principle, be extended to any finite-dimensional setting, not just the classical $\ell_p$ and Schatten classes.

\section{Dimension strikes back. Carath\'eodory-type results and problems}
\subsection{Johnson--Lindenstrauss-type sketching}
\Href{Theorem}{thm:jl_signals_receiver} follows from the following result 
since the measure of $[0,1]^d$ is one. We note that similar results were obtained in \cite{Eskenazis2023}.
\begin{thm}[Johnson--Lindenstrauss sketching for bounded functions]\label{thm:JL_for_functions}
Let $(\Omega,\mu)$ be a probability space and let $f_1,\dots,f_n\colon \Omega\to[-1,1]$ be functions in $L_2(\Omega,\mu)$.
Fix $\varepsilon\in(0,1)$. Then there exists an integer
\[
k \ \le\ C\,\frac{\ln n}{\varepsilon^2}
\]
and vectors $x_1,\dots,x_n\in\R^k$ such that for all $i,j\in[n]$,
\[
\enorm{x_i-x_j}^2- \varepsilon \ \le\ 
\norm{f_i-f_j}_2^2 \ \le\ 
\enorm{x_i-x_j}^2 + \varepsilon .
\]
Here $C>0$ is an absolute constant.
\end{thm}

\begin{proof}
For $\omega\in\Omega$ define the vector
\[
v(\omega)\in \R^{ {n\choose 2}}
\qquad\text{by}\qquad
v(\omega)_{\{i,j\}}:=\abs{f_i(\omega)-f_j(\omega)}^2,
\quad 1\le i<j\le n.
\]
Then $v(\omega)\in[0,4]^{ {n\choose 2}}$, and
\[
\int_\Omega v(\omega)\,d\mu(\omega)
\ =\
\Big(\ \norm{f_i-f_j}_2^2\ \Big)_{1\le i<j\le n}
\ \in\ \R^{ {n\choose 2}}.
\]
In particular, the vector of all pairwise $L_2$-distances squared is the barycenter of the set $\{v(\omega):\omega\in\Omega\}$.

Apply \Href{Lemma}{lem:Caratheodory_sequence_l_p} in the space $\ell_\infty^{ {n\choose 2}}$ (via the substitution from \Href{Proposition}{prp:l_p_trick_bounds})
to obtain points $\omega_1,\dots,\omega_k\in\Omega$ such that the empirical average
\[
\frac{1}{k}\sum_{s=1}^k v(\omega_s)
\]
approximates $\int_\Omega v(\omega)\,d\mu(\omega)$ in $\ell_\infty^{ {n\choose 2}}$ up to error $\varepsilon$.
Equivalently, for all $i<j$,
\[
\enorm{
\frac{1}{k}\sum_{s=1}^k \abs{f_i(\omega_s)-f_j(\omega_s)}^2
-
\norm{f_i-f_j}_2^2
}
\le \varepsilon.
\]
Now define
\[
x_i:=\frac{1}{\sqrt{k}}\big(f_i(\omega_1),\dots,f_i(\omega_k)\big)\in\R^k.
\]
Then
\[
\enorm{x_i-x_j}^2=\frac{1}{k}\sum_{s=1}^k \abs{f_i(\omega_s)-f_j(\omega_s)}^2,
\]
and the desired inequality follows.
\end{proof}
\subsection{Positive semidefinite matrix sparsification}

We would like to point out another problem that might be approachable via our method.

\begin{Problem}\label{prob:psd_sparsification}
Let $A_1,\dots,A_n\in \mathcal{M}_d$ be positive semidefinite and suppose that
\[
\lambda_1A_1+\cdots+\lambda_nA_n=\id_d,
\]
where the coefficients $\lambda_i$ are nonnegative and summing up to one.
Given $\varepsilon\in(0,1)$, determine the smallest $k$ for which there exist indices $i_1,\dots,i_k\in[n]$
and positive coefficients $\beta_1,\dots,\beta_k$ such that
\[
\Bigl\|\beta_1A_{i_1}+\cdots+\beta_kA_{i_k}-\id_d\Bigr\|_{S_\infty}\le \varepsilon.
\]
\end{Problem}

This is a no-dimensional Carath\'eodory-type problem.
Two landmark results frame the landscape.
Rudelson~\cite{rudelson1999random} showed that $C(\varepsilon)\,d\ln d$ terms suffice (via random averaging) for positive semidefinite decompositions without any a priori rank restriction.
Batson, Spielman, and Srivastava~\cite{BSS14} gave a deterministic construction with $Cd$ terms for sums of rank-one projectors,
a breakthrough closely connected to the resolution of the Kadison--Singer problem.

In~\cite{IVANOV2020108684} we showed that, when the ranks are unrestricted, no $\varepsilon$-approximation is possible with fewer than $c(\varepsilon)\,d\ln d$ terms.
In particular, the logarithmic factor appears naturally, and Rudelson's bound is tight up to constants.
Moreover, Rudelson's argument relies on the same conceptual step of passing from $S_\infty$ to $S_p$,
but it does so in a probabilistic way via type.
It would be very interesting to obtain a deterministic analogue of Rudelson's result.

\begin{conj}\label{conj:deterministic_rudelson}
Let $A_1,\dots,A_n\in \mathcal{M}_d$ be positive semidefinite and suppose that
\[
\lambda_1A_1+\cdots+\lambda_nA_n=\id_d,
\]
where the coefficients $\lambda_i$ are nonnegative and summing up to one.
Then for every $\varepsilon\in(0,1)$ there exists a deterministic algorithm that selects indices
$i_1,\dots,i_k\in[n]$ with $k\le C(\varepsilon)\,d\ln d$ and outputs positive coefficients $\beta_1,\dots,\beta_k$
such that
\[
\Bigl\|\beta_1A_{i_1}+\cdots+\beta_kA_{i_k}-\id_d\Bigr\|_{S_\infty}\le \varepsilon.
\]
\end{conj}

As a separate problem, we would like to understand whether a linear number of summands suffices under additional structure.
For example, is $k=Cd$ possible when each $A_i$ is an orthogonal projection of rank~$2$?

\section{Dimension strikes back. Helly-type results and problems}
\label{sec:l1-to-lpprime-regression}
In this section we will obtain \Href{Theorem}{thm:chebyshev_ball} and \Href{Theorem}{thm:quantum_psd_local_to_global} using our trick from  \Href{Section}{sec:trick_dimension_back} and the corresponding no-dimensional Helly-type results. We note that the proofs are almost identical.
\subsection{Chebyshev regression}
\begin{proof}[Proof of \Href{Theorem}{thm:chebyshev_ball}]
Set
\[
f_i(x):=\abs{\iprod{a_i}{x}-b_i},\qquad i\in[m],
\]
and define the convex slabs
\[
K_i:=\braces{x\in\R^d:\ f_i(x)\le r}.
\]
The assumption says that for every $J\subset[m]$ with $\card{J}=k$ there exists
$x_J\in R\ball{}_{\ell_1^d}$ such that $x_J\in \bigcap_{j\in J}K_j$, i.e.
\begin{equation}\label{eq:kwise_intersections_ball_l1}
R\ball{}_{\ell_1^d}\cap \bigcap_{j\in J}K_j\neq \emptyset
\qquad\text{for all } J\subset[m],\ \card{J}=k.
\end{equation}

Set $p=\ln d$ and $p'=\frac{p}{p-1}=\frac{\ln d}{\ln d-1}\in(1,2]$.
By \Href{Proposition}{prp:l_p_trick_bounds}(i), $\norm{x}_{p'}\le \norm{x}_1$ for all $x\in\R^d$; hence
$R\ball{}_{\ell_1^d}\subset R\ball{}_{\ell_{p'}^d}$ and \eqref{eq:kwise_intersections_ball_l1} implies
\begin{equation}\label{eq:kwise_intersections_ball_lpprime}
R\ball{}_{\ell_{p'}^d}\cap \bigcap_{j\in J}K_j\neq \emptyset
\qquad\text{for all } J\subset[m],\ \card{J}=k.
\end{equation}

Applying the no-dimensional Helly theorem in $\ell_{p'}^d$ (cf.\ \Href{Lemma}{lem:rk_bound_lp}),
we obtain a point $x\in R\ball{}_{\ell_{p'}^d}$ such that for every $i\in[m]$, the distance in $\ell_{p'}^d$
 from $x$ to $K_i$ is at most
\[
 R\,r_k(\ell_{p'}^d)
\le 4\sqrt2\,R\,\frac{\sqrt{\ln d-1}}{\sqrt{k}}
\le 4\sqrt2\,R\,\sqrt{\frac{\ln d}{k}}.
\]
Moreover, again by \Href{Proposition}{prp:l_p_trick_bounds}(i),
\[
\norm{x}_1\le e\,\norm{x}_{p'}\le eR,
\]
so $x\in eR\ball{}_{\ell_1^d}$.

Finally, since $a_i\in[-1,1]^d$, we have $\norm{a_i}_\infty\le 1$ and thus for all $u\in\R^d$,
\[
\abs{\iprod{a_i}{u}}\le \norm{u}_1\le e\,\norm{u}_{p'}.
\]
Hence each $f_i$ is $e$-Lipschitz with respect to $\norm{\cdot}_{p'}$, and 
thus
\[
f_i(x)\le r+4e\sqrt2\,R\,\sqrt{\frac{\ln d}{k}}
\le r+21\,R\,\sqrt{\frac{\ln d}{k}}.
\]
Taking the maximum over $i\in[m]$ completes the proof.
\end{proof}

\subsection{Quantum feasibility from local consistency}

\begin{proof}[Proof of \Href{Theorem}{thm:quantum_psd_local_to_global}]
The argument is essentially identical to the proof of \Href{Theorem}{thm:chebyshev_ball}, with $\ell_1^d$ replaced by the trace class $S_1$ and $\ell_{p'}^d$ replaced by the uniformly convex Schatten class $S_{p'}$.

\smallskip
Set
\[
f_i(\rho):=\abs{\iprod{A_i}{\rho}-b_i},\qquad i\in[m],
\]
where the Frobenius (Hilbert--Schmidt) inner product is
\[
\iprod{A}{B}:=\mathrm{Tr}(A^\ast B).
\]
Since each $A_i$ is self-adjoint and $\rho$ is self-adjoint, $\iprod{A_i}{\rho}=\mathrm{Tr}(A_i\rho)\in\R$.
Define the convex sets
\[
K_i:=\braces{\rho\in\mathcal{M}_d:\ \rho\succeq 0,\ \mathrm{Tr}(\rho)=1,\ f_i(\rho)\le t}.
\]
Thus $K_i$ is the set of density matrices satisfying the $i$-th measurement constraint up to tolerance $t$.
The assumption of the theorem says that for every $J\subset[m]$ with $\card{J}=k$,
\begin{equation}\label{eq:kwise_density_feasible}
\bigcap_{j\in J}K_j\neq \emptyset.
\end{equation}

The rest of the proof is omitted.
It consists of applying the same scheme as in the proof of \Href{Theorem}{thm:chebyshev_ball}:
one uses the substitution trick
\[
S_1 \ \longrightarrow\ S_{p'},
\qquad
p'=\frac{\ln d}{\ln d-1}\in(1,2],
\]
together with \Href{Proposition}{prp:l_p_trick_bounds}(ii), and then applies a no-dimensional Helly theorem in the uniformly convex space $S_{p'}$.
The Lipschitz step uses the assumption $\normsch[\infty]{A_i}\le 1$ and the inequality
$\normsch[1]{\Delta}\le e\,\normsch[p']{\Delta}$ from \Href{Proposition}{prp:l_p_trick_bounds}(ii), yielding the stated bound.
\end{proof}

\bibliographystyle{alpha}
\bibliography{../work_current/uvolit}

\end{document}